\documentclass[11pt,reqno]{amsart}
\usepackage{a4wide}
\usepackage{hyperref}
\usepackage{color}
\usepackage{amsfonts,amssymb}
\usepackage{times}
\usepackage{mathrsfs}

\definecolor{green}{rgb}{0,0.5,0}
\definecolor{blue}{rgb}{0,0,1}

\hypersetup{colorlinks,
linkcolor=blue,
filecolor=green,
citecolor=green}

\theoremstyle{plain}

\newtheorem{neu}{}[section]
\newtheorem{Cor}[neu]{Corollary}
\newtheorem*{Cor*}{Corollary}
\newtheorem{Thm}[neu]{Theorem}
\newtheorem*{Thm*}{Theorem}
\newtheorem{Prop}[neu]{Proposition}
\newtheorem*{Prop*}{Proposition}
\theoremstyle{definition}
\newtheorem{Lemma}[neu]{Lemma}
\newtheorem*{Rmk*}{Remark}
\newtheorem{Rmk}[neu]{Remark}

\newtheorem*{Ex*}{Example}

\newtheorem*{Qu*}{Question}

\newtheorem{Def}[neu]{Definition}

\theoremstyle{remark}

\theoremstyle{definition}

\newcommand{\p}{\partial}
\newcommand{\om}{\omega}

\newcommand{\into}{\hookrightarrow}
\newcommand{\pf}{\longrightarrow}

\newcommand{\N}{{\mathbb{N}}}

\newcommand{\R}{{\mathbb{R}}}
\newcommand{\C}{{\mathbb{C}}}

\newcommand{\Id}{\mathrm{Id}}

\newcommand{\Map}{\mathrm{Map}}

\newcommand{\coker}{\mathrm{coker}}
\newcommand{\ind}{\mathrm{ind}}

\newcommand{\dvol}{\mathrm{dvol}}

\renewcommand{\AA}{\mathcal{A}}
\newcommand{\HH}{\mathcal{H}}

\newcommand{\FF}{\mathcal{F}}

\newcommand{\VV}{\mathcal{V}}

\newcommand{\XX}{\mathcal{X}}
\newcommand{\HHH}{\mathscr{H}}

\newcommand{\UU}{\mathscr{U}}

\newcommand{\x}{\times}

\newcommand{\beq}{\begin{equation}}
\newcommand{\beqn}{\begin{equation}\nonumber}
\newcommand{\eeq}{\end{equation}}

\newcommand{\bea}{\begin{equation}\begin{aligned}}
\newcommand{\bean}{\begin{equation}\begin{aligned}\nonumber}
\newcommand{\eea}{\end{aligned}\end{equation}}

\numberwithin{equation}{section}

\begin{document}
\title[Local invariant for scale structures on mapping spaces]{Local invariant for scale structures on mapping spaces}
\author{Jungsoo Kang}
\address{Department of Mathematical Sciences\\
     Seoul National University}
\email{hoho159@snu.ac.kr}

\begin{abstract}
Scale structures were introduced by H. Hofer, K. Wysocki, and E. Zehnder as a new concept of a smooth structure in infinite dimensions. We prove that scale structures on mapping spaces are completely determined by the dimension of domain manifolds. As a consequence, we give a complete description of the local invariant introduced by U. Frauenfelder for mapping spaces. Product mapping spaces and relative mapping spaces are also studied. Our approach is based on the spectral resolution of Laplace type operators together with the eigenvalue growth estimate.
\end{abstract}

\maketitle
\tableofcontents
\section{Introduction}
Scale structures were introduced by H. Hofer, K. Wysocki, and E. Zehnder to give a new concept of a smooth structure in infinite dimensions, see \cite{HWZ1,HWZ2} and the literature cited therein. One of the natural questions for understanding the geography of the new structures is about the existence of (local) invariants. The only local invariant for finite dimensional topological manifolds is the dimension. Apart from the finite dimensional case, there are no invariants for separable Hilbert spaces since all of them are isometric to $\ell^2$. It turned out in \cite{Fr1} that scale Hilbert spaces which are Hilbert spaces equipped with scale structures are separable. So one may think that there would be no invariants on scale Hilbert spaces as well; however, there exists some invariant coming from a nested sequence of $\ell^2$ spaces as studied by U. Frauenfelder \cite{Fr1}. Moreover he introduced fractal structures on scale Hilbert spaces on which he believes the right structure for a general setup of Floer theory. The local invariant he introduced can be expressed by simple formulas for fractal scale Hilbert spaces. In this paper, we focus on mapping spaces which are scale Hilbert manifolds. We show that scale structures on mapping spaces are completely determined by the dimension of domain manifolds. While proving this, we show that mapping spaces are fractal and give a complete description of the local invariants of them. \\[-1.5ex]

\noindent\textbf{Main Theorem.}
\textit{Two mapping spaces $\Map(N_1,M_1)$ and $\Map(N_2,M_2)$ are locally scale isomorphic if and only if $\dim N_1\!=\!\dim N_2$.}\\[-1.5ex]

Loosely speaking, this is the main result of the present paper.  Below we shall explain scale structures of mapping spaces and restate the main theorem precisely in Theorem A and Corollary A. Scale structures of product mapping spaces are also studied in Theorem B by taking advantage of the $*$-operation on fractal scale Hilbert spaces; moreover, these results go through for relative mapping spaces under the mixed boundary condition, see Corollary B. In fact, scale structures on mapping spaces are relevant to the order of elliptic self-adjoint operators as discussed in the appendix.

\begin{Def}\footnote{Our and Frauenfelder's definition of scale Hilbert spaces is somewhat different from Hofer-Wysocki-Zender's. Their definition only requires that the zeroth level $H_0$ is a Hilbert space.}
A {\em scale smooth structure} on a Hilbert space $H$ is a tuple
$$
\HH=\big\{(H_k,\langle\cdot,\cdot\rangle_k)\big\}_{k\in\N_0}
$$
where $(H_k,\langle\cdot,\cdot\rangle_k)$, $k\in\N_0=\N\cup\{0\}$ are Hilbert spaces and they build a nested sequence
$$
H=H_0\supset H_1\supset H_2\supset\cdots\supset H_\infty:=\bigcap_{k=0}^\infty H_k
$$
with the following two axioms.\\[-1.5ex]
\begin{itemize}
\item[(i)] For each $k\in\N_0$, the inclusion
$$(H_{k+1},\langle\cdot,\cdot\rangle_{k+1})\into (H_k,\langle\cdot,\cdot\rangle_k)$$
 is a compact operator.\\[-1.5ex]
\item[(ii)] The subspace $H_\infty$ is dense in $(H_k,\langle\cdot,\cdot\rangle_k)$ for every $k\in\N_0$.
\end{itemize}
\end{Def}
We will write $\HH^j$ to emphasize that we are dealing with the scale Hilbert space $H_j$ with the scale structure $(H_j)_k=H_{j+k}$ for $j,k\in\N_0$. The scale product of two scale Hilbert spaces $\HH$ and $\HH'$, $\HH\oplus_{sc}\HH'$ is defined by
$$
\big((H\oplus_{sc} H')_k,\langle\cdot,\cdot\rangle_k\big)=\big(H_k\oplus H'_k,\langle\cdot,\cdot\rangle_{H_k}\oplus\langle\cdot,\cdot\rangle_{H'_k}\big).
$$
A scale Hilbert space $\mathcal{Y}:=\{Y_k,\langle\,\cdot,\cdot\,\rangle_k\}_{k\in\N_0}$ is said to be a scale subspace of $\HH$ if $Y_k$ is a subspace of $H_k$ for all $k\in\N_0$. Moreover if $\mathcal{Y}$ is a closed scale subspace of $\HH$, the orthogonal complement of $\mathcal{Y}$ is defined by $\mathcal{Y}^\perp:=\{Y_k^{\perp_{\langle\cdot,\cdot\rangle_k}},\langle\,\cdot,\cdot\,\rangle_k\}$ where $Y_k^{\perp_{\langle\cdot,\cdot\rangle_k}}$ stands for the orthogonal complement of $Y_k$ with respect to $\langle\,\cdot,\cdot\,\rangle_k$. The definition of scale Hilbert manifolds is the obvious modification from the definition of standard manifolds, or see \cite{HWZ1}. The product operation for scale Hilbert manifolds is also defined in a similar vein and denoted by $\x_{sc}$.

\begin{Def}
Let $\HH$ and $\HH'$ be scale Hilbert spaces. A map $T:\HH\to\HH'$ is called a {\em scale operator} if it induces bounded linear operators on each level, i.e. the induced operators
$$
T|_{H_k}:H_k\pf H'_k,\quad k\in\N_0
$$
are bounded and linear. A scale operator $T:\HH\to\HH'$ is said to be a {\em scale isomorphism} if it is invertible, i.e. there exists a scale operator $T^{-1}:\HH'\to\HH$ such that
$$
T^{-1}\circ T=\Id_\HH,\quad T\circ T^{-1}=\Id_{\HH'}
$$
where $\Id_\HH$ and $\Id_{\HH'}$ are scale operators which induce the identity operators on every level. If there is a scale isomorphism between $\HH$ and $\HH'$, then we say that they are {\em scale isomorphic} and denote by $\HH\stackrel{sc}{\cong}\HH'$
\end{Def}

We recall the notion of fractal structures on scale Hilbert spaces studied in \cite{Fr2}. We define a Hilbert space $\ell^2_f$ for a monotone and unbounded function $f:\N\to(0,\infty)$ by
$$
\ell^2_f:=\Big\{x=(x_1,x_2,\cdots)\,\Big|\,x_\mu\in\R,\,\mu\in\N,\,\,\,\sum_{\mu=1}^\infty f(\mu)x_\mu^2<\infty\Big\}
$$
with the inner product
$$
\langle x,y\rangle_f:=\sum_{\mu=1}^\infty f(\mu)x_\mu y_\mu,\quad x,\,y\in\ell_f^2.
$$

We denote by $\widetilde\FF$ the set of functions $f:\N\to(0,\infty)$ being monotone and unbounded. We define the equivalence relation on this space: Two functions $f_1,f_2\in\widetilde\FF$ are called equivalent (write $f_1\sim f_2$) if there exists a constant $c>0$ such that
$$
\frac{1}{c}f_1(\mu)\leq f_2(\mu)\leq cf_1(\mu),\quad \textrm{for all }\;\mu\in\N.
$$
The quotient set of $\widetilde\FF$ by $\sim$ is denoted by
$$
\FF:=\widetilde\FF/\sim=\{[f]\,|\,f\in\FF\}.
$$

\begin{Def}
An scale Hilbert space $\HH$ is {\em fractal} if there exists $f\in\widetilde\FF$ such that $\HH$ is scale isomorphic to the scale Hilbert space $\ell^{2,f}$ given by
$$
\ell^{2,f}:=\big\{(\ell_{f^k}^2,\langle\cdot,\cdot\rangle_{f^k})\big\}_{k\in\N_0}.
$$
\end{Def}
One can easily check that $\ell^{2,f_1}$ and $\ell^{2,f_2}$ are scale isomorphic if $f_1\sim f_2$. In other words, an equivalence class $[f]\in \FF$ determines the structure of fractal scale Hilbert spaces.\\[-1.5ex]

In order to define Frauenfelder's invariant for scale Hilbert spaces, we consider a scale Hilbert pair which consists of a pair
$$
\HH_2=\{(H_0,\langle\cdot,\cdot\rangle_0),(H_1,\langle\cdot,\cdot\rangle_1)\}
$$
such that there exists a compact dense inclusion $H_1\into H_0$. Let
$$
\mathscr{S}_2:=\{\HH_2,\,\dim H_0=\infty\}/\sim_2
$$
where $\sim_2$ stands for the equivalence relation given by scale isomorphisms. It turned out in \cite{Fr1} that there exists a bijection
\bean
\Phi:\FF&\pf \mathscr{S}_2\\
[f]&\longmapsto [(\ell^2,\ell^2_f)]
\eea
In particular, for a scale Hilbert space $\HH$, every Hilbert space $(H_k,\langle\cdot,\cdot\rangle_k)$, $k\in\N_0$ is separable. Since every separable Hilbert space is isometric to $\ell^2$, there is no invariant for separable Hilbert spaces. However, scale Hilbert spaces do have the invariant as Frauenfelder introduced: Let $\mathscr{S}$ be the set of infinite dimensional scale Hilbert spaces modulo scale isomorphisms and $\Lambda$ be the upper triangle of $\N_0\x\N_0$, i.e.
$$
\mathscr{S}:=\{\HH,\,\dim H_0=\infty\}/\sim_2,\quad\Lambda:=\{(i,j)\in\N_0\x\N_0\,|\,i<j\}.
$$
Then the following map can be regarded as an {\em invariant} for scale Hilbert spaces.
$$
\mathfrak{K}:\mathscr{S}\pf \mathrm{Map}(\Lambda,\FF)
$$
defined by for $\HH\in\mathscr{S}$, $(i,j)\in\Lambda$,
$$
\mathfrak{K}([\HH])(i,j)=\Phi^{-1}\bigr([H_i,H_j]\bigr).
$$
This also gives a {\em local invariant} for scale Hilbert manifolds and we use the same symbol $\mathfrak{K}$ for that. The (local) invariant $\mathfrak{K}$ can be computed in fractal scale Hilbert spaces (or manifolds). If $\HH$ is scale isomorphic to $\ell^{2,f}$ for some $f\in\widetilde\FF$, then the local invariant for $\HH$ is of the form $\mathfrak{K}([\HH])(i,j)=[f^{j-i}]$.

\begin{Def}
A scale operator $T$ between two scale Hilbert spaces $\HH$ and $\HH'$ is said to be {\em Fredholm} if the following conditions hold.
\begin{enumerate}
\item $\ker T$ is finite dimensional scale subspace of $\HH$.
\item $\mathrm{im} \,T$ is a closed scale subspace of $\HH'$.
\item $\coker\,T:=\mathrm{im} \, T^\perp$ is a finite dimensional scale subspace of $\HH'$.
\end{enumerate}
The {\em index} of a scale Fredholm operator $T:\HH\to \HH'$ is
$$
\ind\, T:=\dim\ker T-\dim\coker\, T.
$$
\end{Def}

\begin{Rmk}
Since $\dim\ker T<\infty$, there exist closed subspaces $Y_k\subset H_k$, $k\in\N_0$ such that $H_k=\ker T\oplus Y_k$. It was proved that $Y_k$ can be chosen so that $\mathcal{Y}=\{Y_k,\langle\cdot,\cdot\rangle_k\}$ is indeed a scale subspace of $\HH$ and $\HH=\ker T\oplus_{sc} \mathcal{Y}$, see \cite{HWZ1}. It also holds that $\HH'=\mathrm{im}\,T\oplus_{sc} \mathrm{im}\,T^\perp$. Moreover by the open mapping theorem, $T|_\mathcal{Y}:\mathcal{Y}\to \mathrm{im} \,T$ is a scale isomorphism.

From the definition of scale-Fredholm, we can extract the regularity property: If $T:\HH\to \HH'$ is a scale Fredholm operator and there are $e\in H_0$ and $f\in H_j$ for some $j\in\N_0$ such that $Te=f$. Then $e\in H_j$ in fact. See \cite{HWZ1} for the proof.
\end{Rmk}

\begin{Def}
A scale operator $T:\HH^1\to \HH^0$ is called a {\em scale Hessian operator} if it is a scale Fredholm operator of index zero and symmetric, i.e. $\langle T\xi,\zeta \rangle_0=\langle \xi,T \zeta \rangle_0$ for any $\xi,\zeta\in H_1$.
\end{Def}

Frauenfelder gave the following evidence that fractal structure is the right structure for a general setup of Floer theory.

\begin{Thm}\cite{Fr2}\label{thm:thm of Fr}
A scale Hilbert space carrying a scale Hessian operator is fractal.
\end{Thm}

Now we are in a position to describe the main results of this paper. It is well-known that mapping spaces
$$
\Map(N,M)=\bigr\{\big(W^{k+k_0,2}(N,M),\langle\;\cdot\;,\;\cdot\;\rangle_{W^{k+k_0,2}(N,M)}\big)\bigr\}_{k\in\N_0}
$$
considered in various types of Floer theory carry a scale Hessian operator. Here $N$ is a compact Riemannian manifold and $M$ is an arbitrary manifold and $k_0$ is the smallest natural number satisfying $2k_0\!>\!n\!=\!\dim N$. They are scale Hilbert manifolds modeled on the following scale Hilbert spaces, see Proposition \ref{prop:mapping space is sc-manifold}.
$$
\XX(N,u^*TM)=\bigr\{\big(\Gamma^{k+k_0,2}(N,u^*TM),\langle\;\cdot\;,\;\cdot\;\rangle_{W^{k+k_0,2}(N,u^*TM)}\big)\bigr\}_{k\in\N_0}
$$
for $u\in C^\infty(N,M)$. According to Theorem \ref{thm:thm of Fr}, such mapping spaces are expected to have fractal scale structures locally. The above mapping spaces depend on $g$ the metric of $N$, but due to Corollary A below the mapping space with a different metric $g'$ is scale isomorphic to the original space; thus we do not indicate the choice of metrics for notational convenience.  We shall prove that this scale Hilbert space is fractal and moreover, the dimension of the domain manifold $N$ determines fractal scale structures and the local invariant of mapping spaces. The precise statements are given below.\\[-1.5ex]

\noindent\textbf{Theorem A.} \textit{ Let $E$ be a vector bundle over a closed Riemannian manifold $(N,g)$. A scale Hilbert space $\XX(N,E)$ is scale isomorphic to $\ell^{2,f}$ for $f(\mu)=\mu^{2/\dim N}$, $\mu\in\N$. In particular, the invariant $\mathfrak{K}$ is given by }
$$
\mathfrak{K}\,([\XX(N,E)])(i,j)=[\mu^{2(j-i)/\dim N}].\\[.5ex]
$$

This theorem will be proved in Section 3 and Corollary A below is a direct consequence of the theorem. It is worth mentioning that this result shows that components of mapping spaces are locally scale isomorphic.\\[-1.5ex]

\noindent\textbf{Corollary A.}
\textit{In consequence of Theorem A,  the local invariant $\mathfrak{K}$ for $\Map(N,M)$ is
$$
\mathfrak{K}\,([\Map(N,M)])(i,j)=[\mu^{2(j-i)/\dim N}].
$$
Moreover, let $(N_1,g_1)$ and $(N_2,g_2)$ be closed Riemannian manifolds and  $M_1$ and $M_2$ be any manifolds. Then $\dim N_1\!=\!\dim N_2$ if and only if $\Map(N_1,M_1)$ is locally scale isomorhpic to $\Map(N_2,M_2)$.}\\[-1.5ex]

\noindent\textbf{Theorem B.} \textit{Let $E_1$ and $E_2$ be vector bundles over closed Riemannian manifolds $N_1$ and $N_2$ respectively and let $\dim N_1\!\leq\!\dim N_2$. A product of scale Hilbert spaces $\XX(N_1,E_1)\oplus_{sc}\XX(N_2,E_2)$ is scale isomorphic to $\XX(N_1,E_1)$. Accordingly, $\Map(N_1,M_1)\x_{sc}\Map(N_2,M_2)$ is locally scale isomorphic to $\Map(N_2,M_2)$ for arbitrary manifolds $M_1,M_2$.}\\[-1.5ex]

Even if $N$ has nonempty boundary, we can draw the same conclusion as above by imposing the mixed boundary condition which generalizes both Dirichlet and Neumann boundary conditions.\\[-1.5ex]

\noindent\textbf{Corollary B.}
\textit{If a compact manifold $N$ has nonempty boundary, Theorem A, Corollary A, and Theorem B are true under the mixed boundary condition.}\\[-1.5ex]

In the appendix, we discuss the relations between scale structures (and hence the local invariant) of mapping spaces and the order of elliptic self-adjoint operators on elliptic complexes.

\subsubsection*{Acknowledgments}
I am grateful to my advisor Urs Frauenfelder for fruitful discussions. I also thanks to Jeong Hyeong Park for helpful communications.

\section{Preliminaries}

\subsection{Spectral resolution}
Let $(N,g)$ be an $n$-dimensional closed Riemannian manifold and $E$ be a vector bundle over $N$ equipped with a bundle metric $\langle \cdot,\cdot\rangle_E$. We denote the spaces of smooth sections of $E$ resp. $T^*N\otimes E$ by $\Gamma(N,E)$ resp. $\Gamma(N,T^*N\otimes E)$. We denote by $\Gamma^2(N,E)$ the completion of $\Gamma(N,E)$ with respect to the the $L^2$-product given by
$$
\langle\phi,\psi\rangle_{L^2(N,E)}=\int_N \langle\phi,\psi\rangle_E \,\dvol_N,\quad \phi,\psi\in \Gamma(N,E).
$$
We also need the $L^2$-product on $\Gamma(T^*N\otimes E)$
$$
\langle\phi,\psi\rangle_{L^2(N,T^*N\otimes E)}=\int_N \langle \phi,\psi\rangle_{T^*N\otimes E}\,\dvol_N,\quad \phi,\psi\in \Gamma(N,T^*N\otimes E).
$$
where $\langle \cdot,\cdot\rangle_{T^*N\otimes E}$ is the bundle metric on $T^*N\otimes E$ induced by $g$ and $\langle \cdot,\cdot\rangle_E$. If there is no confusion, we shall write $L^2$ instead of $L^2(N,E)$ and $L^2(N,T^*N\otimes E)$. We consider a Riemannian connection $\nabla:\Gamma(N,E)\to\Gamma(N,T^*N\otimes E)$ and take the formal $L^2$-adjoint operator of $\nabla$,
$$
\nabla^*:\Gamma(N,T^*N\otimes E)\to \Gamma(N,E),\quad \langle\nabla \phi,\psi\rangle_{L^2}=\langle \phi,\nabla^* \psi\rangle_{L^2}.
$$
Then the {\em Bochner Laplacian} is defined by
$$
\Delta:=\nabla^*\nabla:\Gamma(N,E)\pf\Gamma(N,E).
$$
This can be equivalently defined by $\Delta\phi=-\mathrm{trace}\nabla^2\phi$ where  $\nabla^2 \phi$ the second covariant derivative of $\phi\in\Gamma(N,E)$  induced by the connection $\nabla$ on $E$ together with the Levi-Civita connection on $T^*N$.\\[-1.5ex]

A real number $\lambda\in\R$ is called an {\em eigenvalue} if there is some nonzero $\phi\in\Gamma(N,E)$ satisfying $\Delta \phi=\lambda\phi$. Such a $\phi\in \Gamma(N,E)$ is called an {\em eigensection} associated to $\lambda$. The set of all eigenvalues of $\Delta$ is called the {\em spectrum} and denoted by
$$
\mathrm{Spec}(N)=\mathrm{Spec}(N,g)=\{\lambda_\mu\}_{\mu\in\N}=\{\lambda_1\leq\lambda_2\leq\cdots \leq\lambda_\mu\leq\cdots\}.
$$
We say that $\{\phi_\mu,\lambda_\mu\}_{\mu\in\N}$ is a {\em discrete spectral resolution} of $\Delta$ if the set $\{\phi_\mu\}_{\mu\in\N}$ is a complete orthonormal basis for $\Gamma^2(N,E)$ where $\phi_\mu\in\Gamma(N,E)$ so that $\Delta \phi_\mu=\lambda_\mu \phi_\mu$.

\begin{Thm}\label{thm:spectral resolution}
Let $\Delta$ be the Bochner Laplacian on $E$. Then the followings hold:
\begin{itemize}
\item[(i)] There exists a discrete spectral resolution of $\Delta$, $\{\phi_\mu,\lambda_\mu\}_{\mu\in\N}$.
\item[(ii)] There are only finitely many non-positive eigenvalues and $\lambda_\mu\sim C\mu^{2/n}$ for some constant $C>0$ as $\mu\to\infty$.
\end{itemize}
\end{Thm}
\begin{proof}
The assertions hold for general elliptic self-adjoint operators of order 2 (e.g. self-adjoint Laplace type operators),  see Theorem \ref{thm:general spectral resolution}. We refer to the Gilkey's book \cite[Chapter 1]{Gi1} or \cite{GLP} for the proof.
\end{proof}

\begin{Rmk}\label{rmk:Laplace-Beltrami operator}
The simplest one among Laplace type operators is the Laplace-Beltrami operator $\Delta_0:C^\infty(N)\to C^\infty(N)$ on smooth function spaces  defined by
$$
\Delta_0u=-\mathrm{div}_g\nabla_g u=-\mathrm{trace}_g\mathrm{Hess}\,u
$$
where $\mathrm{div}_g$, $\nabla_g$, and $\mathrm{Hess}$ stands for the divergence, the gradient, and the Hessian respectively. In this case, the first assertion of the above theorem is proved by examining the Rayleigh quotient and the second assertion is nothing but the Weyl's asymptotic formula, see  \cite{Ber} or \cite{Ch}. The contractible component of a mapping space $\Map(N,M)$ is modeled on a scale Hilbert space $\Map(N,\R^m)$ where $m=\dim M$. If this is the case, the whole arguments of the present paper can be following with the Laplace-Beltrami operator.
\end{Rmk}

\subsection{Equivalence of Sobolev spaces}
A section of vector bundle $E\to N$ is said to be of class $W^{k,p}$ if all its local coordinate representations are in $W^{k,p}$. This definition is independent of the choice of coordinate charts even if $kp\leq n$. But in order to make a definition of maps of class $W^{k,p}$ between manifolds which does not depend on the choice of coordinate charts, we need the following well-known proposition which holds only for $kp>n$, see \cite[Appendix B]{MS}. For such a reason, we only deal with $W^{k,p}$-maps between manifolds for $kp>n$.
\begin{Prop}
Let $\Omega\subset \R^n$ be a bounded open domain with $C^k$ boundary. If $kp>n$ and $\varphi\in C^\infty(\R)$, then we have the following smooth map between Banach spaces.
$$
\bar \varphi_{k,p}:W^{k,p}(\Omega)\pf W^{k,p}(\Omega),\quad \bar \varphi_{k,p}(u)=\varphi\circ u.
$$
\end{Prop}

The following theorem is the Bochner-Weitzenb\"och formula. The Laplace-Beltrami operator in Remark \ref{rmk:Laplace-Beltrami operator} obviously extends to $C^\infty(N,\R^m)$.
\begin{Thm}\label{thm:Bochner-Weitzenboch formula}
Let $\Delta_0$ be the Laplace-Beltrami operator on $N$. Then locally  $\Delta=\Delta_0+R$ where $R$ is an endomorphism of $E$ involving only the curvature tensor.
\end{Thm}
\begin{proof}
The proof can be found in \cite[Chapter 4]{Gi1}.
\end{proof}

Next, we recall a significant estimation for the Laplace-Beltrami operator, called the Calderon-Zygmund inequality.

\begin{Thm}\label{thm:interior regularity}
Let $1<p<\infty$, $k\geq0$ be an integer, and $\Delta_0$ be the Laplace-Beltrami operator on an open domain $\Omega\subset\R^n$. If $u\in C^\infty_c(\Omega)$, then there exists a constant $c>0$ satisfying
$$
||u||_{W^{k+2,p}(\Omega)}\leq c\big(||\Delta_0 u||_{W^{k,p}(\Omega)}+||u||_{L^p(\Omega)}\big).
$$
Accordingly, if $\Delta_0$ be the Laplace-Beltrami operator on a closed manifold $N$ and $u\in C^\infty(N)$, then there exists a constant $c>0$ satisfying
$$
||u||_{W^{k+2,p}(N)}\leq c\big(||\Delta_0 u||_{W^{k,p}(N)}+||u||_{L^p(N)}\big).
$$
Here the constant $c$ depends only on $k,p,$ and $\Omega$ (or $N$).
\end{Thm}
\begin{proof}
The proof can be found in \cite[Chapter 8]{Jo} or \cite[Appendix B]{MS}.
\end{proof}

\begin{Def}
The $\Delta^{k,p}$-norm on $\Gamma(N,E)$ is defined by
\bean
||u||_{\Delta^{k,p}(N,E)}&:=||u||_{L^p}+||\nabla u||_{L^p}+||\Delta u||_{L^p}+|| \nabla\Delta u||_{L^p}\\
&\qquad+\cdots+||\nabla^{2(k/2-\lfloor k/2\rfloor)}\Delta^{\lfloor k/2\rfloor} u||_{L^p}.\\
\eea
Here $||\cdot||_{L^p}$ is either $||\cdot||_{L^p(N,E)}$ or $||\cdot||_{L^p(N,T^*N\otimes E)}$. In particular, the $\Delta^{k,2}$-norm is induced from the $\Delta^{k,2}$-product given by
\bean
\langle u,v\rangle_{\Delta^{k,2}(N,E)}&=\langle u,v\rangle_{L^2}+\langle \nabla u, \nabla v\rangle_{L^2}+\langle \Delta f,\Delta h\rangle_{L^2}+\langle \nabla\Delta u,\nabla\Delta v\rangle_{L^2}\\
&\qquad+\cdots+\langle \nabla^{2(k/2-\lfloor k/2\rfloor)}\Delta^{\lfloor k/2\rfloor} u,\nabla^{2(k/2-\lfloor k/2\rfloor)}\Delta^{\lfloor k/2\rfloor}v\rangle_{L^2}.
\eea
\end{Def}

\begin{Cor}\label{cor:equivalence}
On a vector bundle $E$ over a closed Riemannian manifold $N$, the $W^{k,p}$-norm and the $\Delta^{k,p}$-norm are equivalent for $1<p<\infty$. In particular, the Sobolev spaces defined by each of them coincide.
$$
\Gamma^{k,p}(N,E):=\overline{\Gamma(N,E)}^{\,||\cdot||_{W^{k,p}}}=\overline{\Gamma(N,E)}^{||\cdot||_{\,\Delta^{k,p}}}
$$
\end{Cor}
\begin{proof}
It is easy to see that there exists a constant $c>0$ such that
$$
||\phi||_{\Delta^{k,p}(N,E)}\leq c||\phi||_{W^{k,p}(N,E)},\quad \phi\in \Gamma(N,E),
$$
since locally $\nabla=d+A$ where $d$ is the trivial connection and $A$ is a matrix of 1-forms whose entries are Christoffel symbols. The converse can be shown due to previous theorems. An immediate consequence of the Bochner-Weitzenb\"och formula and the Calderon-Zygmund inequality is that there exists $c_0>0$ satisfying
$$
||\phi||_{W^{k,p}(N,E)}\leq c_0\big(||\Delta \phi||_{W^{k-2,p}(N,E)}+||\phi||_{W^{k-2,p}(N,E)}\big).
$$
Therefore there exist constants $c_0,c_1,\cdots, C>0$ satisfying
\bean
||\phi||_{W^{k,p}}&\leq c_0\big(||\Delta \phi||_{W^{k-2,p}}+||\phi||_{W^{k-2,p}}\big)\\
&\leq c_0\bigr(c_1(||\Delta\Delta \phi||_{W^{k-4,p}}+||\Delta \phi||_{W^{k-4,p}})+c_2(||\Delta \phi||_{W^{k-4,p}}+||\phi||_{W^{k-4,p}})\bigr)\\
&\qquad\vdots\\
&\leq C||\phi||_{\Delta^{k,p}}.
\eea
\end{proof}

\begin{Rmk}
There is an alternative way to prove the preceding corollary. It can be proved that the ${\Delta^{k,p}}$-norm and the norm $||\cdot||_{{\nabla,k,p}}$ defined by for $\phi\in\Gamma(N,E)$,
$$
||\phi||_{{\nabla,k,p}}:=||\phi||_{L^2(N,E)}+||\nabla\phi||_{L^2(N,T^*N\otimes E)}+\cdots+||\nabla^k\phi||_{L^2(N,T^*N^{\otimes k}\otimes E)}
$$
are equivalent by examining the commutator
$$
\nabla^*\nabla\nabla-\nabla\nabla^*\nabla:\Gamma(N,T^*N^{\otimes j}\otimes E)\to \Gamma(N,T^*N^{\otimes j+1}\otimes E),\quad j\in\N.
$$
An advantage of this approach is that the preceding corollary can be proved even for noncompact complete manifolds whose curvature tensors and their covariant derivatives are bounded. See Theorem 1.3 in \cite{Do} (or section 2 in \cite{Sa}) but a wrong identity was used in the proof of \cite{Do}; later on, it was repaired by \cite{Sa}.
\end{Rmk}

\section{Fractal scale structures on mapping spaces}
The objective of this section is to explore the geography of fractal scale structures on a scale Hilbert space $\XX(N,E)$ which consists of
$$
\big(\Gamma^{k+k_0,2}(N,E),\langle\cdot,\cdot\rangle_{W^{k+k_0,2}}\big)\supset \big(\Gamma^{k+k_0+1,2}(N,E),\langle\cdot,\cdot\rangle_{W^{k+k_0+1,2}}\big)\supset\cdots\supset \Gamma(N,E)
$$
where $k_0$ is the smallest natural number satisfying $2k_0>n=\dim N$.

\begin{Thm}\label{thm:fractal str on mapping spaces}
A scale Hilbert space $\XX(N,E)$ is fractal. More precisely, it is scale isomorphic to $\ell^{2,f}$ for $f(\mu)=\lambda_\mu$, $\mu\in\N$ where $\{\lambda_\mu\}_{\mu\in\N}$ is the spectrum of the Bochner Laplacian on $E$.
\end{Thm}
\begin{proof}
According to Theorem \ref{thm:spectral resolution}, $\{\phi_\mu\}_{\mu\in\N}$ eigensections of the Bochner Laplacian $\Delta$  form an $L^2$-orthonormal basis for $\Gamma^2(N,E)$. Let $\lambda_\mu\in\R$ be an eigenvalue associated to $\phi_\mu$, $\mu\in\N$. We can take $\Delta^{k,2}$-product instead of $W^{k,2}$-product due to Corollary \ref{cor:equivalence}. It is easy to see that $\{\phi_\mu\}_{\mu\in\N}$ form a $\Delta^{k,2}$-orthogonal basis for $\Gamma^{k,2}(N,E)$, $k\in\N$ as well: For $i,\,j\in\N$, we compute
\bean
\langle \phi_i,\phi_j\rangle_{\Delta^{k,2}}&=\langle \phi_i,\phi_j\rangle_{L^2}+\langle \nabla \phi_i,\nabla \phi_j\rangle_{L^2}+\langle \Delta \phi_i,\Delta \phi_j\rangle_{L^2}+\langle \nabla\Delta \phi_i,\nabla\Delta \phi_j\rangle_{L^2}\\
&\quad+\cdots+\langle \nabla^{2(k/2-\lfloor k/2\rfloor)}\Delta^{\lfloor k/2\rfloor} \phi_i,\nabla^{2(k/2-\lfloor k/2\rfloor)}\Delta^{\lfloor k/2\rfloor}\phi_j\rangle_{L^2}\\
&=\langle \phi_i,\phi_j\rangle_{L^2}+\langle \Delta \phi_i,\phi_j\rangle_{L^2}+ \lambda_i\lambda_j\langle  \phi_i,\phi_j\rangle_{L^2}\\
&\quad+\cdots +(\lambda_i\lambda_j)^{\lfloor k/2\rfloor}\langle \Delta^{2(k/2-\lfloor k/2\rfloor)} \phi_i,\phi_j\rangle_{L^2}\\
&=(1+\lambda_i+\lambda_i^2+\cdots+\lambda_i^k)\langle \phi_i,\phi_j\rangle_{L^2}\\
&=(1+\lambda_i+\lambda_i^2+\cdots+\lambda_i^k)\delta_{ij}.
\eea
Let $f(\mu)=\lambda_\mu$, $\mu\in\N$ and consider the following map between two scale Hilbert spaces.
\bean
\Phi:\XX(E,N)&\pf\ell^{2,f}\\
\psi&\longmapsto \Big(\cdots,\frac{1}{\sqrt{\sum_{j=0}^{k_0}\lambda_\mu^j}}(\psi,\phi_\mu)_{\Delta^{k_0,2}},\cdots\Big)
\eea
Then the map $\Phi$ is a scale isomorphism since for $\lambda_\mu\geq 1$,
$$
\lambda_\mu^{k_0}\leq1+\lambda_\mu+\lambda_\mu^2+\cdots+\lambda_\mu^{k_0}\leq (1+k_0)\lambda_\mu^{k_0}.
$$
\end{proof}

\begin{Rmk}
In the case of $\Map(S^1,\R)$ whose levels are
$$
\big(L^2(S^1,\R),\langle\cdot,\cdot\rangle_{L^2}\big)\supset \big(W^{1,2}(S^1,\R),\langle\cdot,\cdot\rangle_{W^{1,2}}\big)\supset\cdots\supset C^\infty(S^1,\R),
$$
the usual Fourier basis forms a $W^{k,2}$-orthogonal basis for $W^{k,2}(S^1,\R)$, $k\in\N$ as well as an $L^2$-orthonormal basis for $L^2(S^1,\R)$. The previous theorem together with Theorem \ref{thm:spectral resolution}  yield that $\Map(S^1,\R)$ is scale isomorphic the $\ell^{2,f}$ for $f(\mu)=\mu^2$, $\mu\in\N$.
This also can be shown by a straightforward computation with the Fourier basis as well.
\end{Rmk}

We have not introduced the notion of differential of functions or maps  in the scale-world since it is not our main concern; and we refer to \cite{HWZ1}. The following is a useful criterion for scale-smoothness.

\begin{Thm}\cite{HWZ1}\label{thm:sc-smooth}
Let $\HH$ and $\HH'$ be scale Hilbert spaces and $\VV$ be an open subset of $\HH$. Assume that a map $T:\VV\to \HH'$ is scale continuous and $T|_{V_{m+k}}:V_{m+k}\to H'_m$ is of class $C^{k+1}$ for all $m,k\geq 0$. Then $T$ is scale smooth.
\end{Thm}

\begin{Prop}\label{prop:mapping space is sc-manifold}
The mapping space $\Map(N,M)$ is a scale Hilbert manifold with local charts on $\XX(N,u^*TM)$ for $u\in C^\infty(N,M)$.
\end{Prop}
\begin{proof}
The proof immediately follows from the Eliasson's work \cite{El} together with the previous theorem. We first pick a smooth map $u\in C^\infty(N,M)$; since $C^\infty(N,M)$ is dense in $W^{k,2}(N,M)$ for all $k\in\N$, it suffices to find open covering charts near smooth maps. We denote by the bundle map $\tilde{u}:u^*TM\to TM$ induced by $u:N\to M$. Let $D_\epsilon TM$ be the $\epsilon$-disk subbundle of $TM$. There exists a small $\epsilon>0$ such that the exponential map $\exp_{|D_\epsilon T_pM}$ is a diffeomorphism onto an open neighborhood of $p\in M$. Then we have the following parametrization:
\bean
\exp_u^k:\Gamma^{k+k_0,2}(N,u^*D_\epsilon TM)&\pf W^{k+k_0,2}(N,M)\\
\phi(\cdot)&\longmapsto \exp\big(\tilde{u}(\phi(\cdot))\big)
\eea
It turns out that $W^{k+k_0,2}(N,M)$ is a Hilbert manifold and its differentiable structure is given by $\{\UU^k_u,(\exp_u^k)^{-1}\}_{u\in C^\infty(N,M)}$ where $\UU_u^k:=\exp_u^k(\Gamma^{k+k_0,2}(N,u^*D_\epsilon TM))$, see \cite{El,Kl}. In order to prove that $\Map(N,M)$ is a scale Hilbert manifold, we take a close look at the following map.
$$
\mathrm{Exp}_{uu'}:\VV\pf \XX(N,u'^*TM)
$$
where $\VV$ is an open subset in $\XX(N,u'^*TM)$ given by  $\VV_k:=(\exp_u^{k})^{-1}(\UU^k_u\cap\UU^k_{u'})$ and
$$
\mathrm{Exp}_{uu'}|_{V_k}:=(\exp_{u'}^k)^{-1}\circ\exp^k_u:(\exp_u^k)^{-1}(\UU^k_u\cap\UU^k_{u'})\pf (\exp^k_{u'})^{-1}(\mathscr{U}^k_u\cap\UU^k_{u'}).
$$
Due to \cite{El}, we know that each $\mathrm{Exp}_{uu'}|_{V_k}$ for all $k\in\N$ is of class $C^\infty$ and $\mathrm{Exp}_{uu'}$ is obviously scale continuous. Thus, applying Theorem \ref{thm:sc-smooth}, we prove that $\mathrm{Exp}_{uu'}$ is scale smooth for all $u,u'\in C^\infty(N,M)$ and hence the proposition.
\end{proof}

\noindent\textbf{Proof of Theorem A.} According to Theorem \ref{thm:fractal str on mapping spaces}, $\XX(N,E)$ is scale isomorphic to $\ell^{2,f}$ for $f(\mu)=\lambda_\mu$, $\mu\in\N$. As $\mu\to\infty$, $\lambda_\mu$ is asymptotically converge to $\mu^{2/n}$ due to Theorem \ref{thm:spectral resolution}. Therefore $\XX(N,E)$ is scale isomorphic to $\ell^{2,f}$ for $f(\mu)=\mu^{2/n}$, $\mu\in\N$.
\hfill$\square$\\[-1.5ex]

\noindent\textbf{Proof of Corollary A.} The proof follows from Theorem A and Proposition \ref{prop:mapping space is sc-manifold} \hfill$\square$\\[-1.5ex]

In what follows, we shall define an operation on $\FF$ to study fractal scale structures of product mapping spaces. See the introduction for definitions of sets $\FF$ and $\widetilde\FF$. The $*$-operation on $\widetilde\FF$ is defined to be for $f,h\in\widetilde\FF$,
\bea
f*h(1)&=\min\{f(\mu),h(\mu)\,|\,\mu\in\N\}\\
f*h(2)&=\min\big\{\{f(\mu),h(\mu)\,|\,\mu\in\N\}\setminus\{f*h(1)\}\big\}\\
\vdots\\
f*h(i)&=\min\big\{\{f(\mu),h(\mu)\,|\,\mu\in\N\}\setminus\{f*h(1),\dots,f*h(i-1)\}\big\}\\
\vdots
\eea
\begin{Lemma}\label{lemma:equiv}
There exists a constant $c>0$ such that for $f,f',h\in\widetilde\FF$ if $f\sim f'$,
$$
\frac{1}{c}f'*h(\mu)\leq f*h(\mu)\leq cf'*h(\mu)\quad \textrm{for all }\;\mu\in\N.
$$
\end{Lemma}
\begin{proof}
Since $f\sim f'$, there exists $c>0$ such that
$$
\frac{1}{c}f'(\mu)\leq f(\mu)\leq cf'(\mu),\quad \mu\in\N.
$$
We will show that the assertion holds for this $c>0$. Assume on the contrary that $f*h(\eta)>cf'*h(\eta)$ for some $\eta\in\N$. The only nontrivial case is as follows: $r,\, s\in\N$,
$$
\{f*h(1),\cdots,f*h(\eta)\}=\{f(1),\cdots,f(r),h(1),\cdots h(\eta-r)\},
$$
$$
\{f'*h(1),\cdots,f'*h(\eta)\}=\{f'(1),\cdots,f'(s),h(1),\cdots h(\eta-s)\}\\[0.5ex].
$$
If $f*h(\eta)=f(r)$, by assumption, $f(r)>cf'*h(\eta)\geq cf'(s)$ and thus $r>s$. This implies that $\eta-r<\eta-s$ and $h(\eta-r)\leq h(\eta-s)$. But then $f(r)\leq h(\eta-s)$ and this leads to a contradiction
$$
f*h(\eta)=f(r)\leq h(\eta-s)\leq f'*h(\eta).
$$

Suppose that $f*h(\eta)=h(\eta-r)$, then by assumption, $h(\eta-r)>cf'*h(\eta)\geq ch(\eta-s)$. Thus $r<s$ and $f(s)\geq h(\eta-r)$. But then we get
$$
f*h(\eta)=h(\eta-r)\leq f(s)\leq cf'(s)\leq cf'*h(\eta)
$$
which also contradicts to our assumption. Thus we have proved that $f*h(\mu)\leq cf'*h(\mu)$ for all $\mu\in\N$. In a similar way, one can prove $f*h(\mu)>\frac{1}{c}f'*h(\mu)$ for all $\mu\in\N$ and this completes the proof.
\end{proof}

This lemma yields that the $*$-operation descends to $\FF$: $[f]*[h]:=[f*h]$. This product operation is commutative and associative. We  endow a partial order on $\FF$ as follows: $[f_1]\leq[f_2]$ if there is $c>0$ such that $f_1(\mu)\leq cf_2(\mu)$ for all $\mu\in\N$. Then $*$-operation preserves this partial order, i.e. $[f_1]*[h]\leq [f_2]*[h]$ if $[f_1]\leq [f_2]$. Moreover it holds that $[f_1]*[h]\leq [h], [f_1]$ interestingly. If we allow $\FF$ to include an element $e(\mu)=\infty$ for all $\mu\in\N$, then $(\FF,*)$ becomes a partially ordered commutative monoid with the identity element $e$.

\begin{Prop}\label{prop:idempotent}
An element $[f]\in\FF$ which can be represented as a polynomial is an idempotent element with respect to the $*$-operation.
\end{Prop}
\begin{proof}
There is no loss of generality in assuming that $f(\mu)=\mu^k$. We note that
$$
f*f(\mu)=f\Big(\Big\lfloor\frac{\mu-1}{2}\Big\rfloor+1\Big)
$$
and thus we have
$$
\Big(\frac{\mu}{2}\Big)^k\leq f*f(\mu)\leq\Big(\frac{\mu+1}{2}\Big)^k.
$$
Since $\mu\in\N$, $(4\mu)^k\geq(\mu+1)^k$, and hence
$$
\Big(\frac{1}{2}\Big)^kf(\mu)=\Big(\frac{1}{2}\Big)^k\mu^k\leq f*f(\mu)\leq 2^k\mu^k=2^kf(\mu).
$$
This implies that $[f]*[f]=[f]$ in $\FF$ and thus the proposition is proved.
\end{proof}

\begin{Lemma}\label{lemma:product of fractal is fractal}
A product of fractal scale Hilbert spaces is fractal again.
\end{Lemma}
\begin{proof}
It suffices to show that the product of $\ell^{2,f_1}$ and $\ell^{2,f_2}$ is scale isomorphic to $\ell^{2,h}$ for some monotone unbounded function $h:\N\to(0,\infty)$. An element in $(\ell^{2,f_1}\oplus_{sc}\ell^{2,f_2})_k=\ell_{f_1^k}^{2}\oplus\ell_{f_2^k}^{2}$ is a sequence $(x,y):\N\to\R\x\R$ such that
$$
\sum_{\mu=1}^\infty f_1^k(\mu)x_\mu^2+f_2^k(\mu)y_\mu^2<\infty,\,\quad x=(x_\mu)_{\mu\in\N},\,\,y=(y_\mu)_{\mu\in\N}.
$$
Then the following map is a scale isomorphism by definition of $*$-operation.
\bean
\ell^{2,f_1}\oplus_{sc}\ell^{2,f_2}&\pf\ell^{2,f_1*f_2}\\
(x,y)&\longmapsto z
\eea
where $z(\mu):=x(j)$ resp. $:=y(j)$ if $f_1*f_2(\mu)=f_1(j)$ resp. $f_2(j)$ for $j\in\N$.
\end{proof}

\noindent\textbf{Proof of Theorem B.}
Due to Theorem \ref{thm:fractal str on mapping spaces}, there exist a scale isomorphism
$$
\XX(N_1,E_1)\oplus_{sc}\XX(N_2,E_2)\stackrel{sc}{\cong}\ell^{2,f_1}\oplus_{sc}\ell^{2,f_2}.
$$
for $f^i(\mu)=\lambda^i_\mu$ where $\{\lambda^i_\mu\}_{\mu\in\N}$ is the spectrum of the Bochner Laplacians on $\Gamma(N_i,E_i)$ for $i\in\{1,2\}$.  Then Lemma \ref{lemma:product of fractal is fractal} yields that
$$
\XX(N_1,E_1)\oplus_{sc}\XX(N_2,E_2)\stackrel{sc}{\cong}\ell^{2,f_1}\oplus_{sc}\ell^{2,f_2}\stackrel{sc}{\cong}\ell^{2,f_1*f_2}.
$$
In addition, due to Theorem \ref{thm:spectral resolution}, $f_1(\mu)\sim\mu^{2/n_1}$ and $f_2(\mu)\sim\mu^{2/n_2}$ as $\mu\to\infty$  where $n_1=\dim N_1$ and $n_2=\dim N_2$. Let us assume that $n_1\leq n_2$, i.e. $[f_1]\geq[f_2]$. Since $[f_2]*[f_2]=[f_2]$ by Proposition \ref{prop:idempotent}, we have
$$
[f_2]=[f_2]*[f_2]\leq [f_1]*[f_2]\leq [f_2].
$$
This shows that $\ell^{2,f_1*f_2}$ and $\ell^{2,f_2}$ are scale isomorphic and hence the theorem is proved:
$$
\XX(N_1,E_1)\oplus_{sc}\XX(N_2,E_2)\stackrel{sc}{\cong}\ell^{2,f_1}\oplus_{sc}\ell^{2,f_2}\stackrel{sc}{\cong}\ell^{2,f_1*f_2}\stackrel{sc}{\cong}\ell^{2,f_2}\stackrel{sc}{\cong}\XX(N_2,E_2).
$$
\hfill$\square$

\section{Relative mapping spaces}
This section is devoted to study relative mapping spaces. Let $(N,\p N)$ be an n-dimensional compact manifold with nonempty boundary. In the presence of boundary, most part of spectral theory continues to work with nice boundary conditions. Here we consider the mixed boundary condition which generalizes both Dirichlet and Neumann boundary conditions. Let $L_p$'s be submanifolds in $M$ parametrized by $p\in\p N$ and denote by $L:=\bigsqcup_{\p N}L_p$. We are interested in relative mapping spaces of the following form.
$$
\Map\big((N,\p N),(M,L)\big)=\big\{\big(W_\p^{k+k_0,2},\langle\cdot\,,\,\cdot\rangle_{W^{k+k_0,2}}\big)\big\}_{k\in\N_0}
$$
where $k_0$ is the smallest natural number satisfying $2k_0> n$ as before and $W_\p^{k,2}$'s are Hilbert manifolds given by
$$
W_\p^{k,2}:=\big\{u\in W^{k,2}(N,M)\,\big|\,u(p)\in L_p,\,\p_\nu u(p)\in N_{u(p)}L_p,\,p\in\p N\big\}.
$$
Here $N_{u(p)}L_p$ is the normal bundle of $L_p\subset M$ at $u(p)$ and $\nu$ stands for the outward pointing unit normal vector field of $N$ at $\p N$. This scale Hilbert manifold is modeled on the following scale Hilbert space.
$$
\XX\big((N,\p N),(u^*TM,u^*TL)\big)=\big\{\big(\Gamma_\p^{k+k_0,2},\langle\cdot\,,\,\cdot\rangle_{W^{k+k_0,2}}\big)\big\}_{k\in\N_0}
$$
where Hilbert spaces $\Gamma_\p^{k,2}$'s are given by
$$
\Gamma_\p^{k,2}:=\big\{\phi\in \Gamma^{k,2}(N,u^*TM)\,\big|\,\phi(p)\in T_{u(p)}L_p,\,\p_\nu \phi(p)\in N_{u(p)}L_p,\,p\in\p N\big\}.
$$
As we mentioned, this kind of boundary condition is said to be the mixed boundary condition; $\phi$ satisfies the Dirichlet boundary condition on $N_{u(p)}L_p$ and the Neumann boundary condition on $T_{u(p)}L_p$, i.e.
$$
\phi(p)|_{N_{u(p)}L_p}=0\quad\&\quad \p_\nu\phi(p)|_{T_{u(p)}L_p}=0,\quad p\in\p N.
$$
Of course $L_p's$ can be a single submanifold, i.e. $L_p=L_q$ for all $p,q\in\p N$. But in Floer theory, boundary points map to different Lagrangian submanifolds in general, see \cite{Fl} for Lagrangian Floer homology and see the end of the first section in \cite{HNS} for Hyperk\"ahler Floer homology with the Lagrangian boundary condition.

When we prove Theorem A, Theorem \ref{thm:spectral resolution} and Theorem \ref{thm:interior regularity} played crucial roles. Corresponding theorems go through for relative mapping spaces under the mixed boundary condition.

\begin{Thm}\label{thm:boundary}
Let $\Delta$ be the Bochner Laplacian on $\bigcap_{k\in\N}\Gamma^{k,2}_\p$. Then the followings hold:
\begin{itemize}
\item[(i)] $\Delta$ is self-adjoint with respect to $L^2$-metric.
\item[(ii)] There exists a discrete spectral resolution of $\Delta$ for $\Gamma^{0,2}_\p$, $\{\phi_\mu,\lambda_\mu\}_{\mu\in\N}$.
\item[(iii)] There are only finitely many non-positive eigenvalues and $\lambda_\mu\sim C\mu^{2/n}$ for some constant $C>0$ as $\mu\to\infty$.
\end{itemize}
\end{Thm}
\begin{proof}
See \cite[Chapter 1]{Gi2} with \cite{Gr} or \cite[Theorem 2.8.4]{GLP}.
\end{proof}

\begin{Thm}\label{thm:bdry2}
Let $(N,\p N)$ be a compact manifold with nonempty boundary. There exists a constant $c>0$ such that for $u\in C^\infty(N)$ with either $u|_{\p N}\equiv 0$ or $\p_\nu u|_{\p N}\equiv0$,
$$
||u||_{W^{k+2,2}(N)}\leq c\big(||\Delta_0 u||_{W^{k,2}(N)}+||u||_{W^{k,2}(N)}\big).
$$
\end{Thm}
\begin{proof}
The proof can be found in \cite{Jo2} and \cite{We}.
\end{proof}
\noindent\textbf{Proof of Corollary B.}
Making use of the above two theorems together with the Bochner-Weitzenb\"och formula (Theorem \ref{thm:Bochner-Weitzenboch formula}), the corollary is proved by following through the arguments of section 3.\hfill$\square$\\[-1.5ex]

\noindent\textbf{Lagrangian Floer homology.} In this subsection, we justify the boundary condition described above is reasonable for Lagrangian Floer theory. Here we only consider the simplest case and refer to \cite{Fl,Oh} for more general set-up. Let $I$ be an interval $[0,1]$ and $\om_{\C^n}$ be the standard symplectic structure on $(\C^n,i)$ with the compatible metric $g(\cdot,\cdot)=\om_{\C^n}(\cdot,i\cdot)$. We note that $\om_{\C^n}$ exact, i.e. $\om_{\C^n}=d\lambda$ for some 1-form $\lambda$ on $\C^n$ and that  $L=\R^n\!\x\!\{\mathbf{0}\}$ is a Lagrangian submanifold. We consider the space of $W^{k,2}$-paths, $k\in\N$, satisfying the Lagrangian boundary condition.
$$
\Omega^k(L:\C^n):=\big\{\gamma\in W^{k,2}(I,\C^n)\,\big|\,\gamma(p),\,i\p_t\gamma(p)\in L=T_{\gamma(p)}L,\,
p\in\{0,1\}\;\big\}.
$$
This space carries the following action functional.
$$
\AA:\Omega^1(L:\C^n)\pf \R,\quad \AA(\gamma):=\int_{I}\gamma^*\lambda.
$$
It is worth noting that the boundary condition, $i\p_t\gamma(p)\in L$, does not make any trouble to do Lagrangian Floer homology with $\AA$.
Next we consider the $W^{k,2}$-tangent bundle, $k\in\N_0$, along $\gamma\in \Omega^1(L:\C^n)$.
$$
\Omega_\gamma^k=\Omega_\gamma^k(L:\C^n):=\big\{\xi\in W^{k,2}(I,\gamma^*T\C^n)\,\big|\,\xi(p),\,i\p_t\xi(p)\in L=T_{\gamma(p)}L,\,
p\in\{0,1\}\;\big\}.
$$
This boundary condition yields that $\xi$ satisfies the Dirichlet boundary  condition on the second $\R^n$ factor of $\R^n\x \R^n=\C^n$ and the Neumann boundary condition on the first $\R^n$ ($=L$).
A direct computation shows that the Hessian of $\AA$ at $\gamma$ is given by
\bean
\HHH_\gamma:\Omega_\gamma^{k+1}&\pf \Omega_\gamma^k\\
\xi&\longmapsto i\p_t\xi
\eea
This boundary condition is necessary for the well-definedness of the Hessian $\HHH_\gamma$. At least it need to hold that $\HHH_\gamma[\xi](p)\in L$ and this follows from the boundary condition:
$$
\HHH_\gamma[\xi](p)=i\p_t\xi(p)\in L.
$$

\section{Appendix: Some remarks on the local invariant}
As we have observed in the introduction, the invariant $\mathfrak{K}$ has simple formulas for fractal scale Hilbert spaces that
$$
\mathfrak{K}([\ell^{2,f}])(i,j)=[f^{j-i}].
$$
Thus the growth types of fractal functions $f:\N\to(0,\infty)$ determine the (local) invariant for fractal scale Hilbert spaces (or manifolds). In Theorem A, we gave a complete description of the local invariant $\mathfrak{K}$ for mapping spaces $\Map(N,M)$:
$$
\mathfrak{K}([\Map(N,M)])(i,j)=[\mu^{2(j-i)/\dim N}],\quad \mu\in\N.
$$
In this appendix, we construct mapping spaces which are fractal scale Hilbert spaces and whose fractal functions have different growth types from $\Map(N,M)$ we have considered. Thus we provide concrete examples of fractal scale Hilbert spaces with a variety of the invariant formulas.

For a scale Hilbert space $\HH$ which is scale isomorphic to $\ell^{2,f}$, we set
$$
\HH[j]:=\big\{\big(H_{jk},\langle\cdot,\cdot\rangle_{jk}\big)\big\}_{k\in\N_0},\quad j\in\N.
$$
Then $\HH[j]$ is of course scale Hilbert subspace of $\HH$; furthermore it is scale isomorphic to $\ell^{2,f^j}$. According to this simple observation, we can easily construct mapping spaces with various polynomial growth types.  For instance,
$$
\Map(N,\R)[j]=\big\{\big(W^{jk,2}(N,\R),\langle\cdot,\cdot\rangle_{W^{jk,2}}\big)\big\}_{k\in\N_0}\stackrel{sc}{\cong}\ell^{2,f},\quad f(\mu)=\mu^{2j/\dim N}.
$$

In this simple example, honestly speaking, the growth type of fractal functions of mapping spaces is determined by the growth type of eigenvalues of the Laplace-Beltrami operator; since an  elliptic operator of Laplace type is of order 2, $\Map(N,\R)$ (or $\Map(N,M)$ in general) has the growth type $f(\mu)=\mu^{2/\dim N}$. Thus using the following theorem, we can build mapping spaces whose fractal functions have of arbitrary polynomial growth types.

\begin{Thm}\label{thm:general spectral resolution}
Let $P$ be an elliptic self-adjoint operator on $\Gamma(N,E)$ of order $d>0$. Then the following holds:
\begin{itemize}
\item[(i)] There exists a discrete spectral resolution of $P$ for $\Gamma^2(N,E)$, $\{\phi_\mu,\lambda_\mu\}_{\mu\in\N}$.
\item[(ii)] There are only finitely many non-positive eigenvalues and $\lambda_\mu\sim C\mu^{d/n}$ for some constant $C>0$ as $\mu\to\infty$.
\end{itemize}
\end{Thm}
\begin{proof}
The proof can be found in \cite[Chapter 1]{Gi1}
\end{proof}
Therefore we conclude that scale structures (and hence the invariant) on the spaces of sections of the following form are determined by the order of elliptic self-adjoint operators $P$ and the dimension of domain manifolds $N$:
$$
\XX_P(N,E)=\big\{\big(\Gamma^k_P(N,E),\langle\cdot,\cdot\rangle_{P,k,2}\big)\big\}_{k\in\N_0}
$$
where each level and metric are given by
$$
\Gamma^k_P(N,E)=\{\phi\in\Gamma^2(N,E)\,|\,P^j\phi\in\Gamma^2(N,E),\,1\leq j\leq k\},\quad\langle\phi,\psi\rangle_{P,k,2}=\sum_{j=0}^k\langle P^j\phi,P^j\psi \rangle_{L^2}.
$$
Then following through the argument of the previous sections, we can prove that
$$
\XX_P(N,E)\stackrel{sc}{\cong}\ell^{2,f},\quad f(\mu)=\mu^{\mathrm{ord\,} P/\dim N},
$$
and thus the invariant is of the form
$$
\mathfrak{K}([\XX_P(N,E)])(i,j)=[\mu^{\mathrm{ord\,}P(j-i)/\dim N}].
$$


\vskip20pt
\begin{thebibliography}{99}
\bibitem[Be]{Be} M. Berger, ``A panoramic view of Riemannian geometry'', Springer.
\bibitem[B\'e]{Ber} P.H. B\'erard, ``Spectral geometry: Direct and inverse problems'', Lecture Notes in Mathematics, Springer.
\bibitem[Ch]{Ch} I. Chavel, ``Eigenvalues in Riemannian geometry'',  Academic Press, 1994.
\bibitem[Do]{Do} J. Dodziuk, {\em Sobolev spaces of differential forms and deRham-Hodge isomorphism}, J. Diff. Geom., \textbf{16} (1981) 63--73.
\bibitem[El]{El} H. Eliasson, {\em Geometry of manifolds of maps}, J. Diff. Geom. \textbf{1} (1967) 169--194.
\bibitem[Fl]{Fl} A. Floer, {\em Morse theory for Lagrangian intersections}, J. Diff. Geom., \textbf{28} (1988), 513--547.
\bibitem[Fr1]{Fr1} U. Frauenfelder, {\em First steps in the geography of scale Hilbert structures,} (2009), arXiv:0910.3980.
\bibitem[Fr2]{Fr2} U. Frauenfelder, {\em Fractal scale Hilbert spaces and scale Hessian operators,} (2009), arXiv:0912.1154.
\bibitem[Gil]{Gi1} P. Gilkey, ``Invariance theory, the Heat equation, and the Atiyah-Singer index theorem'', Studies in Advanced Mathematics.
\bibitem[Gi2]{Gi2} P. Gilkey, {\em The spectral geometry of operators of Dirac and Laplace type}, ``Handbook of Global analysis'', Elsevier.
\bibitem[GLP]{GLP} P. Gilkey, J. Leahy, J.H. Park, {\em Spinors, spectral geometry, and Riemannian submersions}, Seoul National University, Research Institute of Mathematics, Global Analysis Research Center, 1998.
\bibitem[Gr]{Gr} G. Grubb, ``Functional calculus of pseudodifferential boundary problems'', Progress in Mathematics, \textbf{65} Birkh\"auser, 1996.
\bibitem[HNS]{HNS} S. Hohloch, G. Noetzel, D. Salamon, {\em Hypercontact structures and Floer theory}, Geom. Topol., \textbf{13} (2009), 2543--2617.
\bibitem[HWZ1]{HWZ1} H. Hofer, K.Wysocki, E. Zehnder, ``Fredholm theory in polyfolds I: Functional analytic methods'', Book in preparation.
\bibitem[HWZ2]{HWZ2} H. Hofer, K.Wysocki, E. Zehnder, {\em A general Fredholm theory I: A splicing-based differential geometry}, J. Eur. Math. Soc. \textbf{9} (2007), 841--876.
\bibitem[Jo1]{Jo} J. Jost, ``Partial differential equations'', Springer, 2002.
\bibitem[Jo2]{Jo2} J. Jost, ``Postmodern Analysis'', Springer, 2005.
\bibitem[Ka]{Ka} J. Kang, {\em Morse lemma on scale Hilbert spaces}, in preparation.
\bibitem[Kl]{Kl} W. Klingenberg, ``Lectures on closed geodesics'', Springer, Berlin, 1978.
\bibitem[MS]{MS} D. McDuff, D. Salamon, ``J-holomorphic curves and symplectic topology'', American Mathematical Society Colloquium Publications, \textbf{52} 2004.
\bibitem[Oh]{Oh} Y.-G. Oh, {\em Floer cohomology of Lagrangian intersections and pseudo-holomorphic disks I}, Comm. Pure Appl. Math. 46 (1993), 949-994.
\bibitem[Sa]{Sa} G. Salomonsen, {\em Equivalence of Sobolev spaces}, Results Math., \textbf{39} (2001), 115--130.
\bibitem[We]{We} K. Wehrheim, ``Uhlenbeck Compactness'', EMS Series of Lectures in Mathematics, 2004.
\end{thebibliography}
\end{document}